\documentclass[10pt]{amsart}
\usepackage{texmaX,semmaX,semtkzX}
\usepackage{makecell}
\usepackage{verbatim}
\usepackage{comment}

\begin{document}

\title{A note on local formulae for the parity of Selmer ranks}
\author{Adam Morgan}

\address{University of Glasgow, Glasgow, G12 8QQ, UK.}
\email{adam.morgan@glasgow.ac.uk}

\subjclass[2010]{11G40 (11G10, 11G20, 11G30, 14K15)}

\begin{abstract}
In this note, we provide evidence for a certain `twisted' version of the parity conjecture for Jacobians, introduced in prior work of  V. Dokchitser, Green, Konstantinou and the author. To do this, we use  arithmetic duality theorems for abelian varieties to study  the determinant of certain endomorphisms acting on $p^\infty$-Selmer groups. 
\end{abstract}

\maketitle

\section{Introduction}

The principal aim of this paper is to provide evidence for a conjecture, made in the work \cite{DGKM2024} of  V. Dokchitser, Green, Konstantinou and the author, concerning a certain `twisted' version of the parity conjecture for Jacobians of curves. Our main result in this direction, which was advertised as Theorem 1.8 in \cite{DGKM2024}, is Theorem \ref{thm:termcompat_into} below. We obtain it as a consequence of Theorem \ref{thm:main_local_formula}, which relates the determinant of the action of a suitable endomorphism of an abelian variety on its $p^\infty$-Selmer group  to auxiliary local determinants. 
As we explain below, this can be viewed as a (partial) generalisation of a formula established in \cite{MR2551757} by Coates, Fukaya, Kato and Sujatha.
 
\subsection{Main results}
Let $K$ be a number field, let $X$ be a curve over $K$, and let $G$ be a finite subgroup of $\textup{Aut}_K(X)$. The group $G$ acts on the Mordell--Weil group $\textup{Jac}_X(K)$ of the Jacobian of $X$, making $\textup{Jac}_X(K)\otimes_{\mathbb{Z}} \mathbb{C}$ into a finite dimensional representation of $G$. If $\tau$ is a finite dimensional orthogonal representation of $G$, one has a corresponding global root number $w(X^\tau)\in \{\pm 1\}$ (see \cite[Section 2]{DGKM2024}). Conjecture~1.4(2) of  \cite{DGKM2024}, which is implied by standard conjectures on $L$-functions, is the equality
\begin{equation} \label{parity_version_intro}
w(X^\tau)=(-1)^{\left \langle \tau, \textup{Jac}_X(K)\otimes \mathbb{C}\right \rangle}.\end{equation}
In this paper, we verify a  variant of this conjecture for a large class of representations $\tau$. Specifically, let $\Theta$ be a Brauer relation for $G$. One can associate to $\Theta$,  and the choice of an auxiliary prime $p$,  a certain (non-empty) set of  representations $\mathcal{T}_{\Theta,p}$ -- see Section \ref{recollections} for details. These representations are central to the approach to the parity conjecture developed in \cite{DGKM2024}. Specifically, it follows from the main results of \cite{DGKM2024} (see Theorems 1.5 and 1.7 in particular) that establishing \eqref{parity_version_intro} for all curves $X$ and all orthogonal $\tau\in \mathcal{T}_{\Theta,p}$, for varying $\Theta$ and $p$, is enough to deduce the `classical' parity conjecture for all Jacobians.

To state our result, take $X/K$ as above and write $\Omega^1(\textup{Jac}_X/K)$ for the $K$-vector space of global sections of  $\Omega^1_{\textup{Jac}_X/K}$. Denote by $X_p(\textup{Jac}_X)$ the Pontryagin dual of the $p^\infty$-Selmer group of $\textup{Jac}_X$, and write $\mathcal{X}_p(\textup{Jac}_X)=X_p(\textup{Jac}_X)\otimes_{\mathbb{Z}_p} \mathbb{Q}_p$.

\begin{theorem} \label{thm:termcompat_into}
Suppose that $\textup{Jac}_X$ has semistable reduction at all nonarchimedean places of $K$, and that $\Omega^1(\textup{Jac}_X/K)$ is self-dual as a $G$-representation. Let $p$ be an odd prime and suppose that $\textup{Jac}_X$ has good ordinary reduction at all primes above~$p$. Then, for every  orthogonal representation $\tau \in \mathcal{T}_{\Theta,p}$, we have
$$
  w(X^{\tau})=(-1)^{\left\langle \tau,\mathcal{X}_p(\textup{Jac}_X)\right\rangle}.
$$
In particular, if the Shafarevich--Tate group of $\textup{Jac}_X$ is finite, then \eqref{parity_version_intro} holds for  $\tau$. 
\end{theorem}

\begin{remark}
For the application to the `classical' parity conjecture alluded to above, it suffices to   consider only those curves $X$ and subgroups $G$ of $\textup{Aut}_K(X)$ such that $\Omega^1(\textup{Jac}_X)$ is self-dual as a $G$-module; see  \cite[Theorem 1.7]{MR2534092}. Thus the assumption on  $\Omega^1(\textup{Jac}_X)$ in Theorem \ref{thm:termcompat_into} is not overly restrictive. 
\end{remark}

\begin{remark}
We also prove the analogue of Theorem \ref{thm:termcompat_into} for the slightly expanded notion of a \textit{pseudo} Brauer relation introduced in \cite[Definition 4.1]{DGKM2024}; see Remark \ref{pseudo_rem}.
\end{remark}

 To describe our second main result, let $A/K$ and $B/K$ be principally polarised abelian varieties, and let $\phi:A\rightarrow B$ be an isogeny. Denote by $\phi^t:B\rightarrow A$ the corresponding dual isogeny and write $f=\phi^t\phi\in \textup{End}(A)$.  
 Given a nonarchimedean place $v$ of $K$ at which $A$ has semistable reduction,  
 denote by $\mathfrak{X}_{A,v}$ the character group of the toric part of the reduction of $A/K_v$. This is a finite rank free $\mathbb{Z}$-module with an action of the Frobenius element $\textup{Fr}_v$ at $v$. 

\begin{theorem} \label{thm:main_local_formula}
 Let $p$ be an odd prime. Suppose that $A$ has semistable reduction at all nonarchimedean places $v$ of $K$, and good ordinary reduction at all places $v\mid p$. Then  
\begin{eqnarray*}
\lefteqn{  \textup{ord}_p\det(f\mid \mathcal{X}_p(A))  \equiv } \\&&  \sum_{v\mid \infty}  \textup{ord}_p \det(f \mid \Omega^1(A/K_v)) +   \sum_{v\nmid \infty} \textup{ord}_p \det(f \mid  \mathfrak{X}_{A,v}^{\textup{Fr}_v}\otimes \mathbb{Q}_p ) \pmod 2. 
\end{eqnarray*}
(For each $v\mid \infty$, we have $\det(f \mid \Omega^1(A/K_v))=\pm \deg(\phi)\in \mathbb{Z}$; see Section \ref{ssec:arch_places}.)
\end{theorem}

\begin{remark} \label{rem:CFKS_parallel}
If $f=\phi^t\phi$ is multiplication by $p$, Theorem \ref{thm:main_local_formula} expresses the parity of the $p^\infty$-Selmer rank $\textup{rk}_p(A/K)$ as a sum of local terms:
\[(-1)^{\textup{rk}_p(A/K)}=\prod_{v\mid \infty}(-1)^{ \dim A}\cdot \prod_{v\nmid \infty}(-1)^{\textup{rk}\mathfrak{X}_{A,v}^{\textup{Fr}_v}}.\]
For each place $v$ of $K$, the contribution from $v$ to the right hand side of this formula is equal to the local root number $w(A/K_v)$ (for nonarchimedean places see \cite[Proposition 3.26]{MR2534092} and for archimedean places see \cite[Lemma 2.1]{MR2309184}).  In particular, Theorem \ref{thm:main_local_formula} immediately implies the $p$-parity conjecture for $A$. Having assumed that $A$ admits an isogeny $\phi$ splitting multiplication by $p$, this was known previously by work of Coates--Fukaya--Kato--Sujatha \cite{MR2551757}, and a direct proof of Theorem \ref{thm:main_local_formula} in this case can be extracted from their work.  The value of Theorem \ref{thm:main_local_formula} is that it applies to more general isogenies. 
\end{remark}

 \subsection{A local formula}
To prove Theorem \ref{thm:main_local_formula}, we first establish a purely local formula, which may be of independent interest, before combining it with standard arithmetic duality theorems to deduce the result. Specifically, we prove the following. In the statement, for a local field $F$, an odd prime $p$, and a character $\chi:G_F=\textup{Gal}(\overline{F}/F)\rightarrow \mathbb{F}_p^{\times}$, we denote by $(-1,\chi)_F\in \{\pm 1\}$ the result of evaluating $\chi$ at $-1\in F^{\times}$ via local class field theory.   

\begin{theorem} \label{main_thm_local}
Let $F$ be a local field of characteristic $0$. Let $A/F$ and $B/F$ be principally polarised abelian varieties, let $\phi:A\rightarrow B$ be an isogeny and let $\phi^t:B\rightarrow A$ be the corresponding dual isogeny. Write $f=\phi^t\phi\in \textup{End}(A)$ and let $p$ be an odd prime. If $F$ is nonarchimedean with residue characteristic different from $p$ then suppose that $A$ has semistable reduction, while if $F$ has residue characteristic $p$ suppose that $A$ has good ordinary reduction. 
Then we have
\[ (-1)^{\textup{ord}_p \frac{\#B(F)/\phi A(F)}{\#A(F)[\phi]}}\cdot \big(-1,\chi_{A[\phi],p}\big)_{F}=   \begin{cases}  (-1)^{\textup{ord}_p \det(f \mid \Omega^1(A/F))}~~&~~F\textup{ archimedean},\\   (-1)^{\textup{ord}_p \det(f \mid \mathfrak{X}_A^{\textup{Fr}}\otimes \mathbb{Q}_p)}~~&~~F \textup{ nonarchimedean} ,\end{cases} \]
where $\chi_{A[\phi],p}:G_{F}\rightarrow \mathbb{F}_p^{\times}$ is the character of Definition \ref{def:the_galois_character}.
 \end{theorem}
 
 \begin{remark}
Under the assumptions on the reduction of $A$ present in Theorem \ref{main_thm_local},  the quantity $\big(-1,\chi_{A[\phi],p}\big)_{F}$ is trivial unless $F\cong \mathbb{R}$ or $F$ has residue characteristic $p$ (see Lemma \ref{chi_unram_trivial_lemma} for the proof). We remark also that, in light of the root number formulae referred to in Remark \ref{rem:CFKS_parallel}, the local formula above is a close analogue of \cite[Theorem 2.7]{MR2551757} and \cite[Theorem 3]{MR2389862}. 
 \end{remark}

\subsection{Layout of the paper}
Sections \ref{sec_2} and \ref{sec_3} contain preliminary material and are purely algebraic in nature. In Section \ref{sec_2} we define and study certain characters associated to finite Galois modules. In particular, we define the character $\chi_{A[\phi],p}$ appearing in the statement  of Theorem \ref{main_thm_local}. Section \ref{sec_3} contains results of a lattice-theoretic nature, with Proposition \ref{main_size_prop} playing a key role in the proof of Theorem \ref{main_thm_local}. In Section \ref{sec_4}, we prove the local formula presented as Theorem \ref{main_thm_local}, drawing on the results of Sections \ref{sec_2} and \ref{sec_3}. In Section \ref{sec_5}, we prove Theorem \ref{thm:main_local_formula} by combining Theorem \ref{main_thm_local} with standard global duality theorems. In Section \ref{sec_6}, we combine Theorem \ref{main_thm_local} with various results from \cite{DGKM2024} to prove Theorem \ref{thm:termcompat_into}. We also prove an auxiliary result that was advertised in \cite{DGKM2024}, comparing the local root numbers $w(X^\tau/K_v)$ defined in that work to  certain `arithmetic' local constants. We present this latter result as  Proposition \ref{prop:tamagawa_regulator}.

\subsection{Acknowledgments}
We would like to thank Alex Bartel, Vladimir Dokchitser, Omri Faraggi, Holly Green and Alexandros Konstantinou for helpful conversations and correspondence.

This work has been supported by the Engineering and Physical Sciences Research Council (EPSRC) grant EP/V006541/1 ‘Selmer groups, Arithmetic Statistics and Parity Conjectures’.
 
\section{Characters defined from finite abelian $p$-groups} \label{sec_2}

Let $p$ be an odd prime, let $n\geq 1$ be an integer, and let $M$ be a finite abelian $p$-group of exponent $p^n$. Consider the $\textup{Aut}(M)$-stable filtration 
\[0\subseteq p^{n-1}M[p^n]\subseteq ...\subseteq pM[p^2]\subseteq M[p]\]
 of $M[p]$. 
 
 \begin{definition} \label{defi_the_character}
 Let $0\leq i\leq n-1$. Define the homomorphism
 \[\chi_{M,i}:\textup{Aut}(M)\longrightarrow \textup{GL}\big(p^iM[p^{i+1}]\big)\stackrel{\textup{det}}{\longrightarrow}\mathbb{F}_p^{\times}.\]
Here $\textup{GL}\big(p^iM[p^{i+1}]\big)$ is the group of $\mathbb{F}_p$-linear automorphisms of $p^iM[p^{i+1}]$, the first map is induced by the natural action of $\textup{Aut}(M)$ on  $p^iM[p^{i+1}]$, and the second map is the determinant. Define the character $\chi_M:\textup{Aut}(M)\rightarrow \mathbb{F}_p^{\times}$ by setting 
\[\chi_M=\prod_{i=0}^{n-1}\chi_{M,i}.\]
\end{definition}

\begin{remark} \label{rem:torsion_filtration}  \label{semisimple_remark}
If $M$ is an $\mathbb{F}_p$-vector space, then $\chi_M$ is the map sending an automorphism to its determinant. More generally, we can consider the filtration 
\begin{equation} \label{torsion_filtration}
0\subseteq M[p]\subseteq M[p^2]\subseteq M[p^3]\subseteq ...\subseteq M[p^n]=M
\end{equation}
of $M$, whose graded pieces are
\[M[p^{i+1}]/M[p^i]\cong p^iM[p^{i+1}].\]
Then $M':=\bigoplus_{i=0}^{n-1}M[p^{i+1}]/M[p^i]$ is a finite dimensional $\mathbb{F}_p$-vector space, and $\chi_M$ is the homomorphism sending an automorphism of $M$ to the determinant of its action on $M'$. 
To push this further, write $G= \textup{Aut}(M)$ and view $M$ as an $R:=(\mathbb{Z}/p^n)[G]$-module in the obvious way. We can refine the filtration \eqref{torsion_filtration} to a composition series for $M$ as an $R$-module. If $M_{\textup{ss}}$ denotes the semisimplification of $M$ (i.e. the direct sum of the composition factors), then we see that $M_{\textup{ss}}$ is an $\mathbb{F}_p$-vector space and that, for any $\sigma \in G$, the quantity $\chi_M(\sigma)\in \mathbb{F}_p^{\times}$ is the determinant of the action of $\sigma$ on $M_{\textup{ss}}$. By the Jordan--H{\"o}lder theorem, the isomorphism class of $M_{\textup{ss}}$ as an $R$-module is independent of the choice of composition series, so we may use any composition series to compute $M_{\textup{ss}}$ and, from this, the character $\chi_M$.
\end{remark}

\begin{lemma} \label{order_2_element_lemma}
Let $\sigma \in \textup{Aut}(M)$ have order $2$. Then 
\[\chi_M(\sigma)=(-1)^{\textup{ord}_p\frac{\#M}{\#M^\sigma}}.\]
\end{lemma}

\begin{proof}
Let $G$ be the subgroup of $\textup{Aut}(M)$-generated by $\sigma$. As above, for each $1\leq i\leq n-1$, we have a short exact sequence of $G$-modules
\[0\longrightarrow M[p^i]\longrightarrow M[p^{i+1}]\stackrel{p^i}{\longrightarrow} p^iM[p^{i+1}]\longrightarrow 0.\]
Since $G$ has order coprime to $p$, these sequences remain exact after taking $G$-invariants. Thus
\[\frac{\#M}{\#M^{\sigma}}=\frac{\#M[p^n]}{\#H^0(G,M[p^n])}=\prod_{i=0}^{n-1}\frac{\#p^iM[p^{i+1}]}{\#H^0(G,p^iM[p^{i+1}])}.\]
Fix $0\leq i\leq n-1$ and denote by $V$ the $\mathbb{F}_p$-vector space $p^iM[p^{i+1}]$. Since $G$ has order $2$, we have $V=V_{1}\oplus V_{-1}$ where $\sigma$ acts on $V_1$ trivially and on $V_{-1}$ as multiplication by $-1$. Then 
\[(-1)^{\textup{ord}_p\frac{p^iM[p^{i+1}]}{\#H^0(G,p^iM[p^{i+1}])}}=(-1)^{\dim_{\mathbb{F}_p}V_{-1}}=\det(\sigma \mid V)=\chi_{M,i}(\sigma).\]
Taking the product of this equality over $0\leq i\leq n-1$ gives the result. 
\end{proof}

\begin{definition} \label{def:the_galois_character}
When $F$ is a field of characteristic different from $p$, and $M$ is a $G_F=\textup{Gal}(F^{\textup{sep}}/F)$-module, by an abuse of notation we denote by $\chi_M:G_F\rightarrow \mathbb{F}_p^{\times}$ the composition
\[G_F\longrightarrow \textup{Aut}(M)\stackrel{\chi_M}{\longrightarrow}\mathbb{F}_p^{\times},\]
with the first map arising from the action of $G_F$ on $M$. If $M$ is a finite $G_F$-module but is not necessarily a $p$-group, then we define $\chi_{M,p}:G_F\rightarrow \mathbb{F}_p^{\times}$ to be the character $\chi_{M[p^\infty]}$ associated to the $p$-primary part of $M$.
\end{definition}

\begin{lemma} \label{lem:semisimple}
Let $p$ be a prime and suppose we have a short exact sequence 
\[0\rightarrow M_1\stackrel{f_1}{\longrightarrow} M \stackrel{f_2}{\longrightarrow} M_2\rightarrow 0\] of finite $G_F$-modules. Then 
$\chi_{M,p}=\chi_{M_1,p}\cdot \chi_{M_2,p}.$ 
\end{lemma}

\begin{proof}
Replacing $M$ by $M[p^\infty]$ we can assume that $M$ is a $p$-group, say with exponent $p^n$. Let $G$ be the image of $G_F$ in $\textup{Aut}(M)$, so that the sequence in the statement is a short exact sequence of $R=(\mathbb{Z}/p^n)[G]$-modules. Arguing as in Remark \ref{semisimple_remark}, we see that the semisimplification $M_{\textup{ss}}$ of $M$ as an $R$-module is an $\mathbb{F}_p$-vector space, and that, for any $\sigma \in G_F$, the value of $\chi_M(\sigma)$ is equal to the determinant of $\sigma$ acting on $M_{\textup{ss}}$. Since $M_{\textup{ss}}=(M_1)_{\textup{ss}}\oplus (M_2)_{\textup{ss}}$, the result follows.
\end{proof}

\begin{example} \label{mu_example}
Take $M=\oplus_{i=1}^r\mu_{p^{m_j}}$, where $\mu_{p^{m_j}}$ denotes the $G_F$-module of $p^{m_j}$-th roots of unity in the algebraic closure $\overline{F}$. It follows immediately from Lemma \ref{lem:semisimple} that $\chi_M=\chi_{\textup{cyc},p}^t$, where $t=\sum_{i=1}^rm_j$, and $\chi_{\textup{cyc},p}:G_F\rightarrow \mathbb{F}_p^{\times}$ is the mod $p$ cyclotomic character arising from the action of $G_F$ on $\mu_p$.
\end{example}

\section{Preliminary results on lattices} \label{lattice_results} \label{sec_3}

The principal aim of this section is to prove the lattice-theoretic Proposition \ref{main_size_prop}, which forms a key part of the proof of Theorem \ref{main_thm_local}.

\begin{notation}
Let $X_1$ and $X_2$ be abelian groups and let $\phi:X_1\rightarrow X_2$ be a homomorphism with finite kernel and cokernel. We refer to such a map $\phi$ as an isogeny, and   define 
\[z(\phi):=\#\textup{coker}(\phi)/\#\ker(\phi).\]
\end{notation}
The basic properties of the function $z$ are summarised in \cite[Lemma I.7.2]{MR2261462}, though we caution that the function denoted $z$ there is the inverse of ours. We note in particular that if $X_1$ and $X_2$ are finite, then $z(\phi)=|X_2|/|X_1|$. Moreover, the function $z$ is multiplicative in short exact sequences (as follows from the Snake lemma), and given isogenies 
\[X_1\stackrel{\phi}{\longrightarrow}X_2\stackrel{\phi'}{\longrightarrow}X_3,\]
the composition $\phi'\circ \phi$ is an isogeny  and $z(\phi'  \phi)=z(\phi')z(\phi)$.
 
Suppose now that we have finite-rank free $\mathbb{Z}$-modules $\Lambda_A$ and $\Lambda_B$, and an isogeny  $\phi:\Lambda_A \rightarrow \Lambda_B$. Let $\Lambda_A^\vee=\textup{Hom}(\Lambda_A,\mathbb{Z})$, and define $\Lambda_B^\vee$ similarly. Let $\phi^\vee:\Lambda_B^\vee \rightarrow \Lambda_A^\vee$ denote the dual isogeny.

\begin{lemma} \label{basic properties of z}
We have $z(\phi)=z(\phi^\vee)$.
\end{lemma}

\begin{proof}
Fix bases for $\Lambda_A$ and $\Lambda_B$ and let $M$ be the matrix of $\phi$ with respect to these bases. Then the transpose $M^t$ of $M$ is the matrix of $\phi^\vee$ with respect to the dual bases for $\Lambda_A^\vee$ and $\Lambda_B^\vee$. By properties of Smith normal form, we have 
\[\#\textup{coker}(\phi)=|\textup{det}(M)|=|\textup{det}(M^t)|= \#\textup{coker}(\phi^\vee).\]
Since both $\phi$ and $\phi^\vee$ are injective,  the result follows.
\end{proof}

Now suppose that $G$ is a finite cyclic group acting on $\Lambda_A$ and $\Lambda_B$, and suppose that $\phi$ is $G$-equivariant. Write $H^0(\phi)$ for the induced isogeny $\Lambda_A^G\rightarrow \Lambda_B^G$.

\begin{lemma} \label{main prelim lattice lemma}
We have
\[\frac{z(H^0(\phi))}{z(H^0(\phi^\vee)) }=\frac{\#H^1(G,\Lambda_B) }{\# H^1(G,\Lambda_A) }.\]
\end{lemma} 

\begin{proof}
Let $\sigma$ denote a generator of $G$, write $\Delta=\sigma -1$ and $N=\sum_{g\in G}g$. Let $\Lambda \in \{\Lambda_A,\Lambda_B\}$. Note that \[(\Lambda^\vee)^G=\Lambda^\vee[\Delta]=\textup{Hom}\left(\Lambda/\Delta\Lambda,\mathbb{Z}\right).\]  Since $\Lambda[N]/\Delta\Lambda\cong H^1(G,\Lambda)$ is finite, any homomorphism  $\Lambda \rightarrow \mathbb{Z}$ that factors through $\Delta\Lambda$ necessarily  factors through $\Lambda[N]$ also. Thus we have   canonical isomorphisms
\[(\Lambda^\vee)^G\cong  \textup{Hom}\left(\Lambda/\Lambda[N],\mathbb{Z}\right)\cong  \textup{Hom}\left(N(\Lambda),\mathbb{Z}\right).\]
From this, we see that 
\begin{eqnarray*} \label{first step in lattice proof}
z(H^0(\phi^\vee))& =&\#\coker\Big(\phi^\vee:\textup{Hom}\left(N(\Lambda_B),\mathbb{Z}\right) \rightarrow \textup{Hom}\left(N(\Lambda_A),\mathbb{Z}\right)  \Big) \\&\stackrel{\textup{Lemma }\ref{basic properties of z}}{=}&\#\coker\Big(\phi:N(\Lambda_A)\rightarrow N(\Lambda_B)\Big).
\end{eqnarray*}
Now consider the commutative diagram with exact rows 
\[\xymatrix{0\ar[r] & N(\Lambda_A)\ar[r]\ar[d]^{\phi_1} & \Lambda_A^G\ar[d]^{\phi_2}\ar[r] & \widehat{H}^0(G,\Lambda_A)\ar[d]^{\phi_3}\ar[r] & 0\\ 0\ar[r] & N(\Lambda_B)\ar[r]^{} & \Lambda_B^G\ar[r]^{} & \widehat{H}^0(G,\Lambda_B)\ar[r] & 0}\]
in which each vertical map is induced by $\phi$, and where $\widehat{H}^0$ denotes the zeroth Tate cohomology group (defined as the quotient of the kernel of $\Delta$ by the image of $N$).
From this we conclude that 
\[\frac{z(H^0(\phi))}{z(H^0(\phi^\vee)) }=z(\phi_2)/z(\phi_1)=z(\phi_3)=\frac{\#\widehat{H}^0(G,\Lambda_B) }{\# \widehat{H}^0(G,\Lambda_A)}=\frac{\#H^1(G,\Lambda_B) }{\# H^1(G,\Lambda_A) }.\]
Here the first equality follows from the above formula for $z(H^0(\phi^\vee))$. For the third equality, we are using   finiteness of $\widehat{H}^0(G,\Lambda_A)$ and $\widehat{H}^0(G,\Lambda_B)$, and for the final equality we are using that  the Herbrand quotient of $\Lambda_A$ is equal to the Herbrand quotient of $\Lambda_B$. This latter fact follows from consideration of the short exact sequence of $G$-modules 
\[0\longrightarrow \Lambda_A \stackrel{\phi}{\longrightarrow}\Lambda_B\longrightarrow \Lambda_B/{\phi(\Lambda_A)}\longrightarrow 0,\]
 the fact that $\Lambda_B/\phi(\Lambda_A)$ is finite, and \cite[Propositions 10 and 11]{MR0219512}. 
\end{proof}

\begin{notation}
 Suppose further that, for $\star\in \{A,B\}$, we have a $G$-invariant non-degenerate bilinear pairing  
 \[\left \langle ~,~\right \rangle_\star  :\Lambda_\star \times \Lambda_\star \longrightarrow \mathbb{Z}.\] 
This pairing induces a homomorphism $\iota_\star:\Lambda_\star \rightarrow \Lambda_\star^\vee$, sending $\lambda \in \Lambda_\star$ to 
$ \left \langle -,\lambda \right \rangle_\star.$
We denote by
$\Phi_{\Lambda_\star}=\Lambda_\star^\vee/\Lambda_\star$
  the cokernel of this map. This is a finite abelian group often referred to as the discriminant group of the lattice, and inherits a natural $G$ action. Following  \cite[Definition 1.4.1]{MR3933907}, we also define the finite abelian group \[\mathfrak{B}_{\Lambda_\star}=\textup{im}\Big(H^1(G, \Lambda_\star)\stackrel{\iota_\star}{\longrightarrow} H^1(G,\Lambda_\star^\vee)\Big).\]
\end{notation}
The key fact we will use about this latter group is the following.  

\begin{proposition} \label{order_of_betts_group}
Let $p$ be an odd prime. Then for  $\star\in \{A,B\}$, we have
\[\textup{ord}_p \#\mathfrak{B}_{\Lambda_\star}\equiv 0 \pmod 2.\]
\end{proposition}
 
 \begin{proof}
This is a formal consequence of \cite[Proposition 2.2.2]{MR3933907}, which shows that $\mathfrak{B}_{\Lambda_\star}$ admits a $\mathbb{Q}/\mathbb{Z}$-valued antisymmetric perfect bilinear pairing.
 \end{proof}

With $\phi:\Lambda_A\rightarrow \Lambda_B$ a $G$-equivariant isogeny as above, suppose we have a (necessarily $G$-equivariant) isogeny $\phi^t:\Lambda_B\rightarrow \Lambda_A$  such that $\phi$ and $\phi^t$ are adjoints for the pairings above. That is, such that 
\[\left \langle \phi(x),y\right \rangle_B=\left \langle x,\phi^t(y)\right \rangle_A\]
for all $x\in \Lambda_A$ and $y\in \Lambda_B$. A simple computation shows that 
\begin{equation} \label{eq:duals_isog_composition}
\phi^\vee \circ \iota_B\circ \phi=\iota_A\circ  \phi=\phi^t\phi.
\end{equation}
(We remark that this identity determines $\phi^t$ uniquely, since both $\phi$ and $\iota_A$ become invertible after tensoring by $\mathbb{Q}$.)

\begin{proposition} \label{main_size_prop}
For each prime $p$, we have
\[\textup{ord}_p\frac{\#\Phi_{\Lambda_A}^G  }{\#\Phi_{\Lambda_B}^G }\equiv \textup{ord}_p \textup{ det}\big(\phi^t \phi\mid \Lambda_A^G\big)+\textup{ord}_p\frac{\#\mathfrak{B}_{\Lambda_A} }{\# \mathfrak{B}_{\Lambda_B} }\pmod 2.\]
In particular, if $p$ is odd then (by Proposition \ref{order_of_betts_group}) we have 
\[ \textup{ord}_p\frac{\#\Phi_{\Lambda_A}^G  }{\#\Phi_{\Lambda_B}^G }\equiv \textup{ord}_p ~\textup{det}\big(\phi^t \phi\mid \Lambda_A^G\big)   \pmod 2.\]
\end{proposition}

\begin{proof}
Taking cohomology of the short exact sequence of $G$-modules
\[0\longrightarrow \Lambda_A \stackrel{\iota_A}{\longrightarrow} \Lambda_A^\vee \longrightarrow \Phi_{\Lambda_A}\longrightarrow 0, \]
 we find
\begin{equation*}
\#\Phi_{\Lambda_A}^G =z\big(H^0(\iota_A)\big) \cdot \frac{\# H^1(G,\Lambda_A)}{\#\mathfrak{B}_{\Lambda_A}}.
\end{equation*}
Combining this, the corresponding result for $A$ replaced by $B$, and  Lemma \ref{main prelim lattice lemma}, gives
\begin{equation} \label{tamagawa invariants lemma}
\frac{\#\Phi_{\Lambda_A}^G  }{\# \Phi_{\Lambda_B}^G }=\frac{z\big(H^0(\iota_A)\big)}{z\big(H^0(\iota_B)\big)} \cdot \frac{z(H^0(\phi^\vee)) }{\# z(H^0(\phi))}\cdot \frac{\#\mathfrak{B}_{\Lambda_B}}{\#\mathfrak{B}_{\Lambda_A} }.
\end{equation}
It follows from (restricting to $G$-invariants in) \eqref{eq:duals_isog_composition} that we have 
\[\frac{z(H^0(\iota_A))}{z(H^0(\iota_B))}=\frac{z(H^0(\phi^\vee))\cdot z(H^0(\phi))}{ z(H^0(\phi^t\phi))}.\]
Substituting into \eqref{tamagawa invariants lemma} we see that, for each prime $p$, we have 
\begin{eqnarray*}
\textup{ord}_p\frac{\#\Phi_{\Lambda_A}^G  }{\# \Phi_{\Lambda_B}^G }  \equiv \textup{ord}_pz(H^0(\phi^t\phi))+\textup{ord}_p \frac{\#\mathfrak{B}_{\Lambda_A}}{\#\mathfrak{B}_{\Lambda_B} }\pmod 2.
\end{eqnarray*}
 Now $H^0(\phi^t\phi)$ is a self-isogeny of the finite-rank free $\mathbb{Z}$-module $\Lambda_A^G$. By properties of Smith normal form, its cokernel has size $|\det (\phi^t\phi\mid \Lambda_A^G)|$.
\end{proof}

\section{Proof of the local formula} \label{sec_4}
 
In this section we prove Theorem \ref{main_thm_local}. Specifically, the result follows from Corollary \ref{cor:arch_local_thm} (archimedean local fields), Proposition \ref{prop:res_not_p_arch_local_formula} (residue characteristic different from $p$) and the combination of Corollary \ref{cor:uniform_residue_char} and Proposition \ref{main_eq_res_prop} (residue characteristic $p$).

For the rest of this section, let $A$ and $B$ be abelian varieties defined over a local field $F$ of characteristic $0$, and let $\phi:A\rightarrow B$ be an isogeny. Suppose that both $A$ and $B$ are principally polarised, and denote by $\phi^t:B\rightarrow A$ the corresponding dual isogeny. Write $f=\phi^t\phi \in \textup{End}(A)$, noting that $f=f^\dagger$, where $\dagger$ denotes the Rosati involution on $\textup{End}(A)$. For each odd prime $p$ we associate to $\phi$ the character  $\chi_{A[\phi],p}:G_F\rightarrow \mathbb{F}_p^{\times}$, as in Definition \ref{def:the_galois_character}. 

\subsection{Archimedean local fields} \label{ssec:arch_places}

Suppose that $F$ is archimedean.

\begin{proposition} \label{main_arch_prop}
Let $p$ be an odd prime. Then we have 
\[(-1)^{\textup{ord}_p\frac{\#B(F)/\phi A(F)}{\#A(F)[\phi]}}=\big(-1,\chi_{A[\phi],p}\big)_F\cdot(-1)^{\textup{ord}_p\textup{deg}(\phi)}.\]
\end{proposition}

\begin{proof}
The case $F=\mathbb{C}$ is trivial, so assume $F=\mathbb{R}$. The group of $\mathbb{R}$-valued points $A(\mathbb{R})$ can naturally be viewed as a real Lie group. Let $A^0(\mathbb{R})$ denote its connected component of the identity. Then $A^0(\mathbb{R})$ is a divisible abelian group, while $A(\mathbb{R})/A^0(\mathbb{R})$ is an elementary abelian $2$-group. Along with the corresponding statements for $B$ and the fact that $p$ is odd, this implies that 
\[\textup{ord}_p\#B(\mathbb{R})/\phi A(\mathbb{R})=\textup{ord}_p \# B^0(\mathbb{R})/\phi A^0(\mathbb{R})=0.\]
Thus
\[(-1)^{\textup{ord}_p\frac{\#B(F)/\phi A(F)}{\#A(F)[\phi]}}\cdot (-1)^{\textup{ord}_p\textup{deg}(\phi)}=(-1)^{\textup{ord}_p\frac{\#A[\phi]}{\#A[\phi]^\sigma}},\]
where $\sigma$ denotes complex conjugation. The result now follows from Lemma \ref{order_2_element_lemma}.
\end{proof}

\begin{cor} \label{cor:arch_local_thm}
 Theorem \ref{main_thm_local} holds for $F$ archimedean. 
\end{cor}

\begin{proof}
Given Proposition \ref{main_arch_prop}, it suffices to show that 
\begin{equation} \label{degree_as_det}
\det\big(f\mid \Omega^1(A/F)\big)=\pm\deg(\phi),
\end{equation}
for which we can assume that $F=\mathbb{C}$. As an abelian variety over $\mathbb{C}$, $A$ has a uniformisation as a complex torus 
$A(\mathbb{C})=V/\Lambda,$
where $V=\Omega^1(A/F)^\vee$ is a complex vector space of dimension $g$, and $\Lambda=H_1(A,\mathbb{Z})$ is a lattice of rank $2g$. The polarisation on $A$ corresponds to a positive definite Hermitian form $H:V\times V\rightarrow \mathbb{C}$. Since $f=f^\dagger$, the endomorphism $f$ is self-adjoint for $H$ (see \cite{MR861973}, Proposition in \S4). In particular, the determinant of $f$ on $V$ is real. 

We now compute (cf. \cite{MR861973}, Lemma in \S2, for the first equality)  
\[\det(f\mid \Lambda)=\deg(f)=\deg(\phi^t\phi)=\deg(\phi)^2.\]
Now $\Lambda\otimes_{\mathbb{Z}}\mathbb{R}= V$, so by \cite[Theorem A.1]{MR0222054} we have
\[\det(f\mid \Lambda)=N_{\mathbb{C}/\mathbb{R}}\big(\det(f\mid V) \big)= \det(f\mid V) ^2.\]
 Thus $\det(f\mid V)=\pm \textup{deg}(\phi)$ as desired.
\end{proof}

\subsection{Nonarchimedean local fields}\label{Sec_lattice}  

Suppose now that $F$ is nonarchimedean, with residue field $k$ and ring of integers $\mathcal{O}_F$. Denote by $v_F:F^\times \twoheadrightarrow \mathbb{Z}$ the normalised valuation on $F$, and denote by $\mid \cdot \mid_F$ the normalised absolute value.

 Let $\mathcal{A}/\mathcal{O}_F$ (resp. $\mathcal{B}/\mathcal{O}_F$) denote the N\'{e}ron model of $A$ (resp. $B$). Denote by $c(A/F)$ (resp. $c(B/F)$) the Tamagawa number of $A$ (resp. $B$). Write $g=\dim A$, and let $\mathcal{F}_A$, $\mathcal{F}_B$ be $g$-dimensional formal group laws over $\mathcal{O}_F$ associated to $\mathcal{A}, \mathcal{B}$ respectively.  The isogeny $\phi:A\rightarrow B$ induces an isogeny $\widehat{\phi}\in \textup{Hom}_{\mathcal{O}_F}(\mathcal{F}_A,\mathcal{F}_B)$.  Denote by $D(\widehat\phi)$ the Jacobian of $\widehat{\phi}$, evaluated at $0$. That is, $\widehat{\phi}$ is a $g$-tuple of power series in variables $\textbf{x}=(x_1,...,x_g)$, with coefficients in $\mathcal{O}_F$, and $D(\widehat{\phi})$ is the $g\times g$ matrix with coefficients in $\mathcal{O}_F$ such that $\phi(\textbf{x})\equiv D(\widehat\phi)\textbf{x} ~\textup{ (mod deg }2).$  The absolute value of the determinant of $D(\widehat \phi)$ is independent of all choices.

\begin{lemma} \label{size_of_coker_tamagawa}
We have  
 \[  \frac{\#B(F)/\phi A(F)}{\#A(F)[\phi]}= \frac{c(B/F)}{c(A/F)}\cdot \big|\textup{det}D(\widehat{\phi})\big|^{-1}_F.\]
\end{lemma}

\begin{proof}
This is  \cite[Lemma 3.8]{MR1370197}.
\end{proof}

For the rest of this section we impose the following assumption.

\begin{assumption}
Assume that $A$, hence also the isogeneous abelian variety $B$, has semistable reduction over $F$.
\end{assumption}

\subsubsection{Component groups and Grothendieck's monodromy pairing}
The results we recall now can be found in \cite{SGA7IX}. We refer to the survey paper of Papikian \cite{MR3204270}, and the precise references therein, for more details.  

Denote by $\textup{Fr}$ the Frobenius element in $\textup{Gal}(\bar{k}/k)$. Let $\Phi_A$ (resp. $\Phi_B$) denote the group (scheme) of connected components of the special fibre of the N\'{e}ron model of $A$ (resp. $B$). The group $\Phi_A(\bar{k})$ is finite abelian and carries a natural continuous $\textup{Gal}(\bar{k}/k)$-action. By definition, with have 
\begin{equation*} \label{Tam_num_defi_eq}
c(A/F)=\#\Phi_A(\bar{k})^{\textup{Fr}}\quad \textup{ and }\quad c(B/F)=\#\Phi_B(\bar{k})^{\textup{Fr}}.
\end{equation*}

Let $T_A$ denote the toric part of the special fibre of  $\mathcal{A}$ (i.e. the maximal subtorus of $\mathcal{A}_k^0$). Let $\mathfrak{X}_A$ denote the character group of $T_A$. It is a finite-rank free $\mathbb{Z}$-module with a continuous $\textup{Gal}(\bar{k}/k)$-action. 
Denote by $A^\vee$ the dual abelian variety. Grothendieck's monodromy pairing gives a bilinear $\textup{Gal}(\bar{k}/k)$-equivariant pairing 
$\mathfrak{X}_A\times \mathfrak{X}_{A^\vee}\rightarrow \mathbb{Z}.$

The given principal polarisation $\lambda:A\rightarrow A^\vee$ induces, by functoriality of N\'{e}ron models, a $\textup{Gal}(\bar{k}/k)$-equivariant isomorphism $ \mathfrak{X}_{A^\vee}\cong\mathfrak{X}_A$, which we denote $\bar{\lambda}$. From this we obtain a pairing 
\begin{equation*} \label{mon_pairing_A}
\left \langle ~,~\right \rangle_A: \mathfrak{X}_A\otimes\mathfrak{X}_A\stackrel{1\otimes \overline{\lambda}^{-1}}{\longrightarrow}\mathfrak{X}_A\otimes \mathfrak{X}_{A^\vee}\longrightarrow \mathbb{Z},
\end{equation*}
where the final map is the monodromy pairing. This pairing is positive definite and symmetric. We define $\left \langle~,~\right \rangle_B$ similarly. 

Write $\mathfrak{X}_A^\vee=\textup{Hom}(\mathfrak{X}_A,\mathbb{Z})$. We have an inclusion 
$\iota_A:\mathfrak{X}_A\rightarrow \mathfrak{X}_A^\vee,$
 sending $x\in \mathfrak{X}_A$ to $\left \langle -,x\right \rangle_A$.  There is then a short exact sequence of $\textup{Gal}(\bar{k}/k)$-modules 
\begin{equation} \label{lattice_cmponent_gp_seq}
0\longrightarrow \mathfrak{X}_A \stackrel{\iota_A}{\longrightarrow}\mathfrak{X}_A^\vee \longrightarrow \Phi_A(\bar{k})\longrightarrow 0.
\end{equation}

The isogeny $\phi:A\rightarrow B$ induces a $\textup{Gal}(\bar{k}/k)$-equivariant homomorphism $\mathfrak{X}_B\rightarrow \mathfrak{X}_A$, which is injective with finite cokernel. Similarly, the dual isogeny $\phi^t:B\rightarrow A$ induces a homomorphism $\mathfrak{X}_A\rightarrow \mathfrak{X}_B$. These maps are adjoints for the pairings $\left \langle~,~\right \rangle_A$ and $\left \langle ~,~\right \rangle_B$. Recall that $f=\phi^t\phi$. This acts as a  $\textup{Gal}(\bar{k}/k)$-equivariant endomorphism of $\mathfrak{X}_A$. 

\begin{proposition} \label{Tam_comp_prop}
Let $p$ be an odd prime. Then we have 
\[\textup{ord}_p~\frac{c(A/F)}{c(B/F)}\equiv \textup{ord}_p \textup{ det}\big(f \mid \mathfrak{X}_A^{\textup{Fr}}\big) \pmod 2.\] 
\end{proposition}

\begin{proof}  
This is a consequence of Proposition \ref{main_size_prop}. To apply that result, in the notation of Section \ref{lattice_results} we take $\Lambda_A$ to be the character group $\mathfrak{X}_A$, equipped with the pairing $\left \langle~,~\right \rangle_A$, and we take $\Lambda_B$ to be the character group $\mathfrak{X}_B$, equipped with the pairing $\left \langle ~,~\right \rangle_B$. Take $G$ to be the Galois group of any finite extension $k'/k$ such that $\textup{Gal}(\bar{k}/k')$ acts trivially on both $\mathfrak{X}_A$ and $\mathfrak{X}_B$. Thus $G$ is a finite cyclic group acting on $\Lambda_A$ and $\Lambda_B$, preserving the pairings. As a result of \eqref{lattice_cmponent_gp_seq}, the corresponding discriminant group  $\Phi_{\Lambda_A}$  is isomorphic to the component group $\Phi_A(\bar{k})$ as a $G$-module, and similarly for $\Phi_{\Lambda_B}$. Finally, we take the maps $\phi$, $\phi^t$ in Proposition \ref{main_size_prop} to be those induced by the isogenies $\phi^t$, $\phi$ respectively. 
\end{proof}

\begin{cor} \label{cor:uniform_residue_char}
Let $p$ be an odd prime. Then we have 
\[(-1)^{\textup{ord}_p\frac{\#B(F)/\phi A(F)}{\#A(F)[\phi]}}=  (-1)^{\textup{ord}_p \textup{ det}(f \mid \mathfrak{X}_A^{\textup{Fr}})} \cdot (-1)^{\textup{ord}_p \big|\textup{det}D(\widehat{\phi})\big|_F}.\]
\end{cor}

\begin{proof}
Combine Proposition \ref{Tam_comp_prop} with Lemma \ref{size_of_coker_tamagawa}.
\end{proof}

\subsubsection{Residue characteristic different from $p$}

As above, $F$ is a nonarchimedean local field of characteristic $0$, with residue field $k$. Recall that $A$ is assumed to be semistable over $F$. For this subsection, we take $p$ to be an odd prime and assume that the characteristic of $k$ is different from $p$.

\begin{lemma} \label{chi_unram_trivial_lemma}
The character
$\chi_{A[\phi],p}:G_F\rightarrow \mathbb{F}_p^{\times}$
is unramified. In particular, we have \[\big(-1,\chi_{A[\phi],p}\big)_F=1.\]
\end{lemma}

\begin{proof}
By definition, $\chi_{A[\phi],p}$ factors through the Galois group of $F(A[p])/F$. Since  $\chi_{A[\phi],p}$ takes values in $\mathbb{F}_p^{\times}$, which has order coprime to $p$, it suffices to show that the inertia subgroup of $\textup{Gal}(F(A[p])/F)$ has order a power of $p$. Since the residue characteristic of $F$ is assumed different from $p$, this is a well-known consequence of the semistability of $A$. That  $\big(-1, \chi_{A[\phi],p}\big)_F=1$ now follows from the fact that, by local class field theory, all units are norms from any fixed unramified extension.
\end{proof}

\begin{proposition} \label{prop:res_not_p_arch_local_formula}
 Theorem \ref{main_thm_local} holds when $F$ is nonarchimedean and has residue characteristic different from $p$. 
\end{proposition}

\begin{proof}  
Since $F$ is assumed to have residue characteristic different from $p$,  we have 
\[\textup{ord}_p \big|\textup{det}D(\widehat{\phi})\big|_F=0.\]
 The result now follows from Corollary  \ref{cor:uniform_residue_char} and Lemma \ref{chi_unram_trivial_lemma}.
\end{proof}

\subsubsection{Residue characteristic equal to $p$}

Suppose for this subsection that $F$ is a finite extension of $\mathbb{Q}_p$.  Recall that $A$ is assumed to be semistable over $F$. By Corollary \ref{cor:uniform_residue_char}, to prove that the local formula in Theorem \ref{main_thm_local} holds for $A$ and $\phi$, it suffices to show that 
\begin{equation} \label{eq:remaining_to_prove}
(-1)^{\textup{ord}_p \big|\textup{det}D(\widehat{\phi})\big|_F}=\big(-1,\chi_{A[\phi],p}\big)_F.
\end{equation}

The (suggested) identity \eqref{eq:remaining_to_prove} makes no reference to the assumed existence of principal polarisations on $A$ and $B$. It is conceivable that it holds for arbitrary (say semistable) isogeneous abelian varieties $A$, $B$. We prove it below only when $A$ is assumed to have good ordinary reduction, though it is likely possible to improve on this using ideas from \cite[Section 2 and Appendix]{MR2551757}. We remark also that if $\phi$ factors as $\phi_1\phi_2$, then both sides of \eqref{eq:remaining_to_prove} split as a product of the corresponding terms for $\phi_1$ and $\phi_2$. In particular, to prove \eqref{eq:remaining_to_prove} we could assume that $A[\phi]\subseteq A[p]$. Having made this reduction, the formula \eqref{eq:remaining_to_prove} for semistable elliptic curves follows immediately from the material in \cite[Sections 5, 6]{MR2389862}.

In the statement of the following proposition, we do not assume that either $A$ or $B$ is principally polarised. 

\begin{proposition} \label{main_eq_res_prop}
Let $A$ be an abelian variety over $F$ with good ordinary reduction, and let $\phi:A\rightarrow B$ be an isogeny. Then we have 
\[(-1)^{\textup{ord}_p \big|\textup{det}D(\widehat{\phi})\big|_F}=\big(-1,\chi_{A[\phi],p}\big)_F.\]
\end{proposition}

\begin{proof}
As above, let $g=\dim A$ and let $\mathcal{F}_A$, $\mathcal{F}_B$ be $g$-dimensional formal group laws over $\mathcal{O}_F$ associated to $A$, $B$ respectively.  Denote by $\mathcal{F}_A(\mathfrak{m})$ the maximal ideal $\mathfrak{m}$ of the ring of integers of $\overline{F}$, equipped with the group structure induced by $\mathcal{F}_A$, and define $\mathcal{F}_B(\mathfrak{m})$ similarly. Write $\ker(\widehat{\phi})$ for the kernel of the map  $\mathcal{F}(\mathfrak{m})\rightarrow \mathcal{G}(\mathfrak{m})$ induced by $\widehat{\phi}$. We have an inclusion of $G_F$-modules $\ker(\widehat{\phi})\subseteq A[\phi]$, and the quotient is unramified (since it injects into $\mathcal{A}_k(\overline{k})$). By Lemma  \ref{lem:semisimple}, we thus have
\begin{equation} \label{red_to_formal_group}
\big(-1,\chi_{A[\phi],p}\big)_F=\big(-1,\chi_{\ker(\widehat{\phi})}\big)_F.
\end{equation}

Note that we have
\begin{equation} \label{eq:version_of_Schaefer}
\textup{ord}_p \big|\textup{det}D(\widehat{\phi})\big|_F=\textup{ord}_p (\#k)^{v_F(\textup{det}D(\widehat{\phi}))}=v_F(\textup{det}D(\widehat{\phi}))\cdot [k:\mathbb{F}_p],
\end{equation}
where $v_F$ denotes the normalised valuation on $F$. 

Denote by $E$ the completion of the maximal unramified extension of $F$, and denote by $\mathcal{O}$ its ring of integers. Since $A$ has good ordinary reduction, so does $B$, and over $\mathcal{O}$ we have isomorphisms of formal groups 
\[\alpha:\mathcal{F}_A\stackrel{\sim}{\longrightarrow}\widehat{\mathbb{G}}_m^g\quad \textup{ and }\quad \beta:\mathcal{F}_B\stackrel{\sim}{\longrightarrow}\widehat{\mathbb{G}}_m^g,\]
where $\widehat{\mathbb{G}}_m$ denotes the formal multiplicative group. Write $\eta=\beta\widehat{\psi}\alpha^{-1}\in \textup{End}_\mathcal{O}(\widehat{\mathbb{G}}_m^g)$.  
The map sending an endomorphism to its Jacobian matrix gives an isomorphism $\textup{End}_\mathcal{O}(\widehat{\mathbb{G}}_m^g)\cong \textup{Mat}_{g}(\mathbb{Z}_p)$ (the case of general $g$ formally reduces to the case $g=1$, which is a consequence of  \cite[IV.1, Theorem 1]{MR242837}). By properties of Smith normal form, we can find invertible matrices $U,V\in \textup{Mat}_g(\mathbb{Z}_p)$ such that  $D(\eta)=U\Lambda V$, where $\Lambda$ is a diagonal matrix with entries in $\mathbb{Z}_p$.  Suppose $\Lambda$ has diagonal entries $(\lambda_1,...,\lambda_g)$. Re-choosing $U$ and $V$ if necessary, we can assume that, for $i=1,...,g$, we have $\lambda_i=p^{n_i}$ for some integers $n_1,...,n_g$. Then $\ker(\widehat{\phi})\cong \oplus_{i=1}^g\mu_{p^{n_i}}$ as an $I_F$-module, where $I_F$ denotes the inertia group of $F$. Let $n=\sum_{i=1}^gn_i$.  Combining the above discussion with Example \ref{mu_example},  we see that 
\begin{equation} \label{eq:D_formula}
v_F(\textup{det}D(\widehat{\phi}))=v_F(\textup{det}D(\eta))=e(F/\mathbb{Q}_p)n,
\end{equation}
 and that the restriction of $\chi_{\ker(\widehat{\phi})}$ to $I_F$ is equal to the restriction of $\chi_{\textup{cyc},p}^n$ to $I_F$, where $\chi_{\textup{cyc},p}$ is the mod $p$ cyclotomic character. In particular, we have 
 \begin{eqnarray*}
 \big(-1,\chi_{\ker(\widehat{\phi})}\big)_F=\big(-1,\chi_{\textup{cyc},p}^n)_F
= \big(-1,\chi_{\textup{cyc},p})^{[F:\mathbb{Q}_p]n}_{\mathbb{Q}_p}=(-1)^{n[F:\mathbb{Q}_p]}.
 \end{eqnarray*}
 The result now follows by combining this equality with  \eqref{red_to_formal_group}, \eqref{eq:version_of_Schaefer} and \eqref{eq:D_formula}.
\end{proof}

\section{Proof of Theorem \ref{thm:main_local_formula}} \label{sec_5}

In this section we turn to the global setting, and deduce Theorem \ref{thm:main_local_formula} by combining Theorem \ref{main_thm_local} with standard global duality theorems. Let $K$ be a number field. 

\begin{notation}
Following  \cite[Definition 4.1]{MR2680426}, for an isogeny $\phi:A\rightarrow B$ between abelian varieties over $K$, we write 
\[Q(\phi)=\#\coker\big( A(K)/A(K)_{\textup{tors}}\stackrel{\phi}{\rightarrow}B(K)/B(K)_{\textup{tors}}\big)\cdot \#\ker\big(\Sha(A)_{\textup{div}}\stackrel{\phi}{\rightarrow} \Sha(B)_{\textup{div}}\big).\]
Here $A(K)_{\textup{tors}}$ denotes the torsion subgroup of $A(K)$, and $\Sha(A)_{\textup{div}}$ denotes the divisible part of the Shafarevich--Tate group of $A$. 
\end{notation}

\begin{remark} \label{global_cokernel_remark}
As observed in \cite[Section 2]{MR2475965}, for each prime $p$, we have
\[\textup{ord}_pQ(\phi)=\textup{ord}_p\#\coker\big(X_p(A)/X_p(A)_{\textup{tors}}\stackrel{\phi}{\rightarrow} X_p(B)/X_p(B)_{\textup{tors}}\big),\]
where $X_p(A)= \textup{Hom}\big(\displaystyle{\lim_\rightarrow}\textup{ Sel}_{p^n}(A), \mathbb{Q}_{p}/\mathbb{Z}_{p}\big)$ is the Pontryagin dual of the $p^\infty$-Selmer group of~$A$. 
\end{remark}

For the rest of this section, we take all the notation from the statement of  Theorem \ref{thm:main_local_formula}. In particular, $\phi:A\rightarrow B$ is an isogeny between principally polarised abelian varieties over $K$, $\phi^t:B\rightarrow A$ is the dual isogeny, and $f=\phi^t\phi\in \textup{End}(A)$. 

\begin{proof}[Proof of Theorem \ref{thm:main_local_formula}]
By Theorem \ref{main_thm_local} and  Artin reciprocity, we have 
\[\prod_{v\mid \infty}(-1)^{ \textup{ord}_p \det(f \mid \Omega^1(A/K_v))} \cdot \prod_{v\nmid \infty} (-1)^{\textup{ord}_p \det(f \mid (\mathfrak{X}_{A,v}\otimes \mathbb{Q}_p)^{\textup{Fr}_v})}=\prod_v(-1)^{\textup{ord}_p \frac{\#B(K_v)/\phi A(K_v)}{\#A(K_v)[\phi]}}.\]
In the proof of \cite[Theorem I.7.3]{MR2261462}, it is shown (using global duality theorems) that 
\[\prod_v \frac{\#B(K_v)/\phi A(K_v)}{\#A(K_v)[\phi]}=\frac{\#B(K)/\phi(A(K))}{\#A(K)[\phi]}\cdot \frac{\#B(K)[\phi^t]}{\#A(K)/\phi^t(B(K))}\cdot \frac{\#\Sha(A)[\phi]}{\#\Sha(B)[\phi^t]}.\]
It is then shown in the proof of \cite[Theorem 4.3]{MR2680426} that the $p$-adic valuaton of the right hand side of the above displayed formula is equal to 
\begin{equation} \label{half_way_there_global_formula}
\textup{ord}_p \frac{Q(\phi^t)}{Q(\phi)}\cdot \frac{\#A(K)_{\textup{tors}}^2}{\#B(K)_{\textup{tors}}^2}\cdot\frac{\#\Sha_0(B)[p^\infty]}{\#\Sha_0(A)[p^\infty]},
\end{equation}
where $\Sha_0(A)$  is the quotient of $\Sha(A)$ by its maximal divisible subgroup, and similarly for~$\Sha_0(B)$. 

Since both $A$ and $B$ are principally polarised and $p$ is assumed odd, standard properties of the Cassels--Tate pairing (see \cite[Theorem 8]{MR1740984}) give 
\[\textup{ord}_p\#\Sha_0(A)[p^\infty]\equiv 0 \equiv \textup{ord}_p\#\Sha_0(B)[p^\infty] \pmod 2.\]
Reducing \eqref{half_way_there_global_formula} modulo $2$ and using \cite[Lemma 4.2 (1)]{MR2680426}, we conclude that 
\[\prod_v(-1)^{\textup{ord}_p \frac{\#B(K_v)/\phi A(K_v)}{\#A(K_v)[\phi]}}=(-1)^{\textup{ord}_pQ(f)}.\]
By Remark \ref{global_cokernel_remark} and properties of Smith normal form, we have 
\[\textup{ord}_pQ(f)=\textup{ord}_p\det\big(f\mid X_p(A/K)\otimes_{\mathbb{Z}_p}\mathbb{Q}_p\big),\]
giving the result.
\end{proof}

\section{Brauer relations in function fields of curves} \label{sec_6}

 In this section, we take $K$ to be a number field and let $X$ be a curve over $K$.\footnote{Following  \cite[Convention 1.1]{DGKM2024}, we assume that $X$ is smooth and proper, but do not assume that $X$ is connected, nor that its connected components are geometrically connected. See \cite{KM2024} for a discussion of the relevant properties  of such curves and their Jacobians.} Let $G$ be a finite subgroup of $\textup{Aut}_K(X)$, so that $G$ acts on the Jacobian $\textup{Jac}_X$ by $K$-automorphisms. 
We implicitly fix an embedding $\overline{\mathbb{Q}_p}\hookrightarrow \mathbb{C}$ for each prime $p$, via which we will view $\overline{\mathbb{Q}_p}[G]$-representations as $\mathbb{C}[G]$-representations without further comment. We denote by $\left \langle ~,~ \right \rangle$ the standard inner product between (characters of) finite dimensional $\mathbb{C}[G]$-representations.  

\subsection{Recollections from \cite{MR2534092} and \cite{DGKM2024}} \label{recollections}
Let $\mathcal{H}$ be a set of representatives for the subgroups of $G$ up to conjugacy. We call an element
$\Theta=\sum_iH_i-\sum_j H_j'\in \mathbb{Z}[\mathcal{H}]$
a \textit{Brauer relation} if 
we have an isomorphism of $\mathbb{C}[G]$-modules
\[ \bigoplus_{i}\mathbb{C}[G/H_i]\cong  \bigoplus_{j}\mathbb{C}[G/H_j'].\]
For example, when $G=S_3$ is the symmetric group of order $6$, the element $2C_2+C_3-2S_3-\{1\}$ is a Brauer relation, where $C_2$ is a choice of subgroup of order $2$ and $C_3$ has order $3$.  

 In the terminology of \cite[Definition 3.16]{DGKM2024}, we say that a $G$-module homomorphism
\[\Phi: \bigoplus_i \mathbb{Z}[G/H_i]\longrightarrow \bigoplus_{j}\mathbb{Z}[G/H_j']\]
\textit{realises} the Brauer relation $\Theta$ if it is injective with finite cokernel (i.e. becomes an isomorphism after tensoring by $\mathbb{C}$).  Per \cite[Lemma 3.17]{DGKM2024}, for every Brauer relation we can find such a homomorphism. If $M$ is a $\mathbb{Z}[G]$-module, then for all subgroups $H$ of $G$, we have a canonical isomorphism $\textup{Hom}_G(\mathbb{Z}[G/H],M)\cong M^H$. In this way, such a $\Phi$ induces a homomorphism 
\[\Phi^*:\bigoplus_jM^{H_j'}\longrightarrow \bigoplus_iM^{H_i}.\]
Writing $S=\bigsqcup_iG/H_i$ and $S'=\bigsqcup_jG/H_j'$, the associated permutation modules $\mathbb{Z}[S]$ and $\mathbb{Z}[S']$ are self-dual, and $\Phi:\mathbb{Z}[S]\rightarrow \mathbb{Z}[S']$ induces a dual $G$-module homomorphism $\Phi^\vee:\mathbb{Z}[S']\rightarrow \mathbb{Z}[S]$. See \cite[Section 3.4]{DGKM2024} for details.

For a field $\mathcal{K}$ of characteristic $0$,   a self-dual $\mathcal{K}[G]$-representation $V$, and a Brauer relation $\Theta$ for $G$, there is an associated  \textit{regulator constant} 
\[\mathcal{C}_\Theta(V)\in \mathcal{K}^{\times}/\mathcal{K}^{\times 2},\]
first introduced by T. and V. Dokchitser in \cite[Definition 2.13]{MR2534092} (see also \cite[Definition 3.6]{DGKM2024}). We now define the set of representations $\mathcal{T}_{\Theta,p}$ appearing in Theorem \ref{thm:termcompat_into}.

\begin{definition}
Let $p$ be a prime, let $E_p$ denote the maximal unramified extension of $\mathbb{Q}_p$, and denote by $\textup{ord}_p:E_p\twoheadrightarrow \mathbb{Z}$ the normalised valuation.  Noting that $\textup{ord}_p$ induces a map $E_p^{\times}/E_p^{\times 2}~\rightarrow \mathbb{Z}/2\mathbb{Z}$, we denote by $\mathcal{T}_{\Theta,p}$ the set of all self-dual $\mathbb{C}[G]$-representations $\tau$ satisfying
\begin{equation} \label{congruence_btw_reps}
\left \langle \tau,V\right \rangle \equiv \textup{ord}_p~\mathcal{C}_\Theta(V) \pmod 2,\end{equation}
for all self-dual $E_p[G]$-representations $V$. 
\end{definition}

\begin{remark}
The definition of the set $\mathcal{T}_{\Theta,p}$   differs slightly from that of the set $\textbf{T}_{\Theta,p}$ introduced in  \cite{MR2534092}, since we require \eqref{congruence_btw_reps} to hold for all self-dual $E_p$-representations, while \cite[Definition 2.50] {MR2534092} only requires it to hold for self-dual $\mathbb{Q}_p$-representations.  Thus, as defined, $\mathcal{T}_{\Theta,p}$ might be a proper subset of $\textbf{T}_{\Theta,p}$. Nevertheless, the same arguments given in \cite[Section 2]{MR2534092} show that $\mathcal{T}_{\Theta,p}$ is non-empty. In particular,
the expression in \cite[Remark 2.51]{MR2534092} defines an explicit element of $\mathcal{T}_{\Theta,p}$, provided $\mathbb{Q}_p$ is replaced by $E_p$ in the direct sum. Moreover, one readily sees that the sets $\mathcal{T}_{\Theta,p}$ and $\textbf{T}_{\Theta,p}$ agree for the Brauer relations for $C_2\times C_2$ and $D_{2p}$ given in \cite[Example 3.14]{DGKM2024}, and for all Brauer relations in symmetric groups. In particular, they agree for all Brauer relations used in the applications to the `classical' parity conjecture given in \cite[Sections 6-8]{DGKM2024}. The advantage of working with the field $E_p$ in what follows is that all $\overline{\mathbb{Q}_p}$-representations with $\mathbb{Q}_p$-valued character are automatically definable over $E_p$ (as is implied by the triviality of the Brauer group of finite extensions of $E_p$; see \cite[Section 12.2]{serre_scott_2014}). This fact is exploited in the proof of Theorem \ref{thm:termcompat} below.
\end{remark}

Finally, we recall that, for an arbitrary orthogonal (complex) representation $\tau$ of $G$, \cite[Section 2.3]{DGKM2024} defines a certain global root number $w(X^\tau)\in \{\pm 1\}$, which decomposes as a product 
\[w(X^\tau)=\prod_{v\textup{ place of }K}w(X^\tau/K_v)\]
of local root numbers $w(X^\tau/K_v)\in \{\pm 1\}$. 

\subsection{Proof of Theorem \ref{thm:termcompat_into}} With the above background in hand, we can now deduce Theorem \ref{thm:termcompat_into} from Theorem \ref{thm:main_local_formula}. This proves a result that was advertised as Theorem 1.8 in \cite{DGKM2024}.

\begin{theorem}[= Theorem \ref{thm:termcompat_into}] \label{thm:termcompat}
Suppose that $\textup{Jac}_X$ has semistable reduction at all nonarchimedean places of $K$, and suppose that $\Omega^1(\textup{Jac}_X)$ is self-dual as a $G$-representation.
 
Let $p$ be an odd prime and suppose that $\textup{Jac}_X$ has good ordinary reduction at all primes above~$p$. Then, for every Brauer relation $\Theta$ for $G$, and every orthogonal  $\tau_{\Theta,p}\in \mathcal{T}_{\Theta,p}$,  we have 
$$
  w(X^{\tau_{\Theta,p}})=(-1)^{\left\langle \tau_{\Theta,p},\mathcal{X}_p(\textup{Jac}_X)\right\rangle}.
$$
\end{theorem}

\begin{proof}
Write $\Theta=\sum_iH_i-\sum_jH'_j$ and let $\Phi$ be a homomorphism realising $\Theta$. For a subgroup $H$ of $G$, denote by $X/H$ the quotient of $X$ by $H$.   By \cite[Theorem 4.3]{DGKM2024}, associated to the $G$-module homomorphism $\Phi$ we have isogeny 
\[f_\Phi: A:=\prod_j \textup{Jac}_{X/H_j'}\longrightarrow B:=\prod_i \textup{Jac}_{X/H_i}.\]
By the same result, we have $f_{\Phi}^t f_\Phi=f_{\Phi \Phi^\vee}$, where $f_{\Phi \Phi^\vee}$ denotes the self-isogeny of $A$ associated to the $G$-module homomorphism $\Phi \Phi^\vee$.

Since $\textup{Jac}_X$ is semistable at all nonarchimedean places $v$ of $K$, the same is true of its homomorphic image $\textup{Jac}_{X/H}$, for each subgroup $H$ of $G$. In particular, the abelian varieties $A$ and $B$ have everywhere semistable reduction. Similarly, both $A$ and $B$ have good ordinary reduction at all places $v$ dividing $p$.

We now apply Theorem \ref{thm:main_local_formula} to the isogeny $f_\Phi:A\rightarrow B$, yielding  
\begin{eqnarray*}
 (-1)^{\textup{ord}_p\det(f_{\Phi \Phi^\vee}\mid \mathcal{X}_p(A))}=  \prod_{v\mid \infty}(-1)^{ \textup{ord}_p \det(f_{\Phi \Phi^\vee} \mid \Omega^1(A/K_v))} \cdot \prod_{v\nmid \infty} (-1)^{\textup{ord}_p \det(f_{\Phi \Phi^\vee} \mid \mathfrak{X}_{A,v}^{\textup{Fr}_v}
\otimes \mathbb{Q}_p)}.
\end{eqnarray*}
To prove the result, it suffices to show that the right hand side of this formula is equal to the global root number $w(X^\tau)$, and that the left hand side is equal to $(-1)^{\langle \tau_{\Theta,p},\mathcal{X}_p(\textup{Jac}_X)\rangle}$.

For each nonarchimedean place $v$ of $K$, there is an isomorphism 
\begin{equation} \label{invariant_identification_lemma}
\mathfrak{X}_{A,v}^{\textup{Fr}_v}\otimes_{\mathbb{Z}}\mathbb{Q}_p\cong \bigoplus_j \big(\mathfrak{X}_{\textup{Jac}_X,v}^{\textup{Fr}_v}\otimes_{\mathbb{Z}}\mathbb{Q}_p\big)^{H_j'},
\end{equation}
identifying  the endomorphism $f_{\Phi \Phi^\vee}$ of $\mathfrak{X}_A^{\textup{Fr}}\otimes_{\mathbb{Z}}\mathbb{Q}_p$ with the endomorphism $(\Phi \Phi^\vee)^*$ of the right hand side. Indeed, this follows from \cite[Remark 4.28]{KM2024}, applied to the contravariant functor sending a semistable abelian variety $A'/K_v$ to the $\mathbb{Q}_p$-vector space $\mathfrak{X}_{A',v}^{\textup{Fr}_v}\otimes_{\mathbb{Z}}\mathbb{Q}_p$.\footnote{Strictly speaking, \cite[Remark 4.28]{KM2024} and the preceding lemma take as input additive functors defined on the category of all abelian varieties over $F$, while the functor considered here is naturally defined only on semistable abelian varieties. One checks, however, that this subtlety makes no difference to the proof.} 

In light of \eqref{invariant_identification_lemma}, it follows from \cite[Corollary 3.21]{DGKM2024} that, for each nonarchimedean place $v$ of $K$, we have
\[\textup{ord}_p\det(f_{\Phi \Phi^\vee} \mid \mathfrak{X}_{A,v}^{\textup{Fr}_v}
\otimes \mathbb{Q}_p)\equiv \mathcal{C}_\Theta(\mathfrak{X}_{\textup{Jac}_X,v}^{\textup{Fr}_v}\otimes_{\mathbb{Z}}\mathbb{Q}_p)\equiv \big \langle \tau_{\Theta,p}, \mathfrak{X}_{\textup{Jac}_X,v}^{\textup{Fr}_v}\otimes_{\mathbb{Z}}\mathbb{C}\big \rangle  \hspace{-5pt}\pmod 2,\]
the second congruence following from the defining property of $\tau_{\Theta,p}$. Consequently, by \cite[Proposition 2.22]{DGKM2024}, we have 
\[ (-1)^{\textup{ord}_p \det(f_{\Phi\Phi^\vee}\mid  \mathfrak{X}_{A,v}^{\textup{Fr}_v}
\otimes \mathbb{Q}_p)}=w(X^{\tau_{\Theta,p}}/K_v)\]
for all nonarchimedean places $v$ of $K$. 
Arguing similarly, noting that the $\mathbb{Q}_p[G]$-representation $\mathcal{X}_p(A/K)$ is self-dual by \cite[Theorem 1.2]{KM2024}, we see that 
\[\textup{ord}_p\det(f_{\Phi \Phi^\vee}\mid \mathcal{X}_p(A))\equiv \left \langle \tau_{\Theta,p}, \mathcal{X}_p(A)\right \rangle  \hspace{-5pt}\pmod 2.\]
Finally, we claim that, for all archimedean places $v$ of $K$, we have 
\begin{equation} \label{omega_eq}
(-1)^{\textup{ord}_p \det(f_{\Phi \Phi^\vee} \mid \Omega^1(A/K_v))}=(-1)^{\left \langle \tau_{\Theta,p},\Omega^1(A/K_v)\right \rangle }.
\end{equation}
To see this, note that the $\mathbb{C}[G]$-representation $ \Omega^1(A/K_v)$ has rational character. Indeed, having assumed that $ \Omega^1(A/K_v)$ is self-dual,  the representation  $ \Omega^1(A/K_v)^{\oplus 2}$ is rational (see, for example, the second proposition on page 82 of \cite{MR861973}). In particular, $ \Omega^1(A/K_v)$ is realisable over the maximal unramified extension of $\mathbb{Q}_p$. We may now deduce \eqref{omega_eq} by arguing precisely as before. 

To complete the proof, it remains to note that the right hand side of \eqref{omega_eq}  is equal to $w(X^{\tau_{\Theta,p}}/K_v)$, by \cite[Proposition 2.22]{DGKM2024}.
\end{proof}

\begin{remark} \label{pseudo_rem}
In \cite[Definition 4.1]{DGKM2024}, the notion of a \textit{pseudo Brauer relation for }$G$ \textit{and} $X$ is defined,  generalising  the notion of a Brauer relation recalled above. Theorem \ref{thm:termcompat} then continues to hold in this setting, with the corresponding representations $\tau_{\Theta,p}$ defined as in \cite[Section 3.3]{DGKM2024}, but with $E_p$ replacing $\mathbb{Q}_p$ again. The proof is identical, with the decision to restrict to Brauer relations in the above made solely for ease of exposition.
\end{remark}

\subsection{A local version of Theorem \ref{thm:termcompat}}
Now let $F$ be a local field of characteristic $0$, let $X$ be a curve over $F$, and let $G$ be a finite subgroup of $\textup{Aut}_F(X)$.  Assume that $\Omega^1(\textup{Jac}_X)$ is self-dual as a $G$-module, and let $\Theta$ be a Brauer relation for $G$. In \cite[Definition 5.16]{DGKM2024}, a local invariant $\Lambda_{\Theta}(X/F)\in \mathbb{Q}$ is defined. This invariant serves as an `arithmetic local root number' in the sense that, if $X$ is instead defined over a number field $K$, then by \cite[Theorem 5.1]{DGKM2024} we have 
\begin{equation}
(-1)^{\left \langle \tau_{\Theta,p},\mathcal{X}_p(\textup{Jac}_X)\right \rangle}=\prod_{v\textup{ place of }K}(-1)^{\textup{ord}_p\Lambda_{\Theta}(X/K_v)},
\end{equation}
for each prime $p$ and each orthogonal representation $\tau_{\Theta,p}\in \mathcal{T}_{\Theta,p}$.

Using methods from the proof of Theorem \ref{thm:termcompat}, we can relate these local invariants to the relevant local root numbers, assuming that $\textup{Jac}_X$ has sufficiently nice reduction. This proves a result that was advertised as Theorem 8.1(2) in \cite{DGKM2024}.

\begin{proposition} \label{prop:tamagawa_regulator}
Let $X/F$ and $\Theta$ be as above. Let $p$ be an odd prime and let $\tau_{\Theta,p}\in \mathcal{T}_{\Theta,p}$ be orthogonal. Assume that either: 
\begin{itemize}
\item[(i)] $F$ is nonarchimedean with residue field $k$, $\textup{Jac}_X$ has semistable reduction, and either $k$ has characteristic different from $p$ or $[k:\mathbb{F}_p]$ is even, or
\item[(ii)] $F\cong \mathbb{C}$.
\end{itemize}
Then we have
\[w(X^{\tau_{\Theta,p}}/F)=(-1)^{\textup{ord}_p\Lambda_{\Theta}(X/F)}.\]
\end{proposition}

\begin{proof}
Let $\Phi$ be a homomorphism realising $\Theta$, and let 
\[f_\Phi: A=\prod_j \textup{Jac}_{X/H_j'}\longrightarrow B=\prod_i \textup{Jac}_{X/H_i}\]
be the corresponding isogeny afforded by \cite[Theorem 4.3]{DGKM2024}.

Suppose that $F\cong \mathbb{C}$.   Arguing as in the proof of Theorem \ref{thm:termcompat}, we have
\[w(X^{\tau_{\Theta,p}}/F)=(-1)^{\textup{ord}_p \det(f_{\Phi}^t f_{\Phi} \mid \Omega^1(A/F))}.\]
By \eqref{degree_as_det}, the right hand side of the above formula is equal to $(-1)^{\textup{ord}_p\textup{deg}(f_\Phi)}$, which is equal to $(-1)^{\textup{ord}_p\Lambda_{\Theta}(X/F)}$ by \cite[Theorem 5.17(4)]{DGKM2024}. 

Now suppose that $F$ is nonarchimedean. This time, arguing as in the proof of Theorem \ref{thm:termcompat} we see that
\[w(X^{\tau_{\Theta,p}}/F)=(-1)^{\textup{ord}_p \det(f_{\Phi}^t f_{\Phi}\mid \mathfrak{X}_{A}^{\textup{Fr}}
\otimes \mathbb{Q}_p)}.\]
Under the assumptions of the statement, the right hand side is then equal to $(-1)^{\textup{ord}_p\Lambda_{\Theta}(X/F)}$,
 by \cite[Theorem 5.17(3)]{DGKM2024} and Corollary \ref{cor:uniform_residue_char}.
\end{proof}

\begin{remark}
In spite of Proposition \ref{prop:tamagawa_regulator}, we expect (e.g. by looking at the case $F=\mathbb{R}$ in Theorem \ref{main_thm_local}) that $w(X^{\tau_{\Theta,p}}/F)$ and $(-1)^{\textup{ord}_p\Lambda_{\Theta}(X/F)}$ need not always be equal. 
\end{remark}

\begin{remark}
As with Theorem \ref{thm:termcompat},  Proposition \ref{prop:tamagawa_regulator} continues to hold for arbitrary pseudo Brauer relations $\Theta$ for $G$ and $X$, provided $\tau_{\Theta,p}$ is defined as in Remark \ref{pseudo_rem}. The proof is identical. 
\end{remark}

\end{document}